\documentclass[11pt]{article} 
\usepackage[utf8]{inputenc} 

\usepackage[margin=1in]{geometry} 
\usepackage{graphicx} 

\usepackage{booktabs} 
\usepackage{array} 
\usepackage{paralist} 
\usepackage{verbatim} 
\usepackage{mathrsfs}
\usepackage{amssymb}
\usepackage{amsthm}
\usepackage{amsmath,amsfonts,amssymb}
\usepackage{esint}
\usepackage{graphics}
\usepackage{enumerate}
\usepackage{mathtools}
\usepackage{xfrac}
\usepackage{subcaption}
\usepackage{stmaryrd}
 \usepackage{mathabx}

\usepackage[usenames,dvipsnames]{xcolor}
\usepackage[colorlinks=true, pdfstartview=FitV, linkcolor=blue, citecolor=blue, urlcolor=blue]{hyperref}
\usepackage[normalem]{ulem}

\usepackage{tikz}
\usetikzlibrary{calc}
\usepackage{pgf}
\usetikzlibrary{external}
\numberwithin{equation}{section}
\numberwithin{figure}{section}

\newtheorem{theorem}{Theorem}[section]

\newtheorem{corollary}[theorem]{Corollary}
\newtheorem{proposition}[theorem]{Proposition}
\newtheorem{lemma}[theorem]{Lemma}
\theoremstyle{definition}
\newtheorem{definition}[theorem]{Definition}

\newcommand*{\N}{\ensuremath{\mathbb{N}}}

\newcommand*{\Z}{\ensuremath{\mathbb{Z}}}

\newcommand*{\R}{\ensuremath{\mathbb{R}}}

\newcommand{\eps}{\varepsilon}

\renewcommand*{\tilde}{\widetilde}

\newcommand{\ep}{\eps}

\DeclareSymbolFont{boldoperators}{OT1}{cmr}{bx}{n}
\SetSymbolFont{boldoperators}{bold}{OT1}{cmr}{bx}{n}
\usepackage{accents}

\newcommand{\T}{\mathbb{T}}
\newcommand{\sol}{\mathcal{T}}

\def\XXint#1#2#3{{\setbox0=\hbox{$#1{#2#3}{\int}$}
\vcenter{\hbox{$#2#3$}}\kern-.5\wd0}}

\let\originalleft\left
\let\originalright\right
\renewcommand{\left}{\mathopen{}\mathclose\bgroup\originalleft}
\renewcommand{\right}{\aftergroup\egroup\originalright}

\newcommand{\C}{\mathbb{C}}

\newcommand{\E}{\mathbb{E}}

\renewcommand{\hat}{\widehat}

\makeatletter
\pgfmathdeclarefunction{erf}{1}{%
  \begingroup
    \pgfmathparse{#1 > 0 ? 1 : -1}%
    \edef\sign{\pgfmathresult}%
    \pgfmathparse{abs(#1)}%
    \edef\x{\pgfmathresult}%
    \pgfmathparse{1/(1+0.3275911*\x)}%
    \edef\t{\pgfmathresult}%
    \pgfmathparse{%
      1 - (((((1.061405429*\t -1.453152027)*\t) + 1.421413741)*\t 
      -0.284496736)*\t + 0.254829592)*\t*exp(-(\x*\x))}%
    \edef\y{\pgfmathresult}%
    \pgfmathparse{(\sign)*\y}%
    \pgfmath@smuggleone\pgfmathresult%
  \endgroup
}
\makeatother

\usepackage{titlesec}

\newcommand{\addperiod}[1]{#1.}
\titleformat{\section}
   {\centering\normalfont\Large}{\thesection.}{0.5em}{}
\titleformat*{\subsection}{\bfseries}
\titleformat{\subsubsection}[runin]
  {\normalfont\bfseries}
  {\thesubsubsection.}
  {0.5em}
  {\addperiod}
\titleformat*{\subsubsection}{\normalfont\itshape}
\titleformat*{\paragraph}{\bfseries}
\titleformat*{\subparagraph}{\large\bfseries}
\newcommand{\Q}{\mathbb{Q}}

\title{A subsequentially fast dynamo on $\T^3$}

\author{
Keefer Rowan\thanks{Courant Institute of Mathematical Sciences,  New York University.
{\footnotesize \href{mailto:keefer.rowan@cims.nyu.edu}{keefer.rowan@cims.nyu.edu}.}
}
}
\date{\today}

\usepackage[nottoc,notlot,notlof]{tocbibind}

\begin{document}

\maketitle

\begin{abstract}
    We construct a smooth velocity field $u$ on $\R_+ \times \T^3$ that exhibits kinematic dynamo action, causing exponential growth in solutions to the magnetohydrodynamic induction equation, with a rate that is uniform in diffusivity, for suitable sequences of diffusivity $\kappa_j \to 0.$ We call this a subsequentially fast dynamo, giving dynamo behavior intermediate between a truly slow dynamo and a truly fast dynamo.
\end{abstract}

\section{Introduction}

A fast dynamo is a divergence-free flow that causes exponential growth in solutions to the magnetohydrodynamic induction equation for the magnetic field uniformly for all sufficiently small magnetic diffusivities. The study of dynamos developed throughout the 20th century, seeking to explain how geophysical and astrophysical bodies can maintain magnetic fields---such as the Earth's magnetic field---on time scales well beyond the natural diffusive time scale associated to the magnetic diffusivity of the body. The beginnings of the explanation were first provided by~\cite{larmor_how_1919}, which provided a qualitative mechanistic argument that the motion of an electrically charged fluid---such as the Earth's molten core---can amplify the magnetic field, causing growth of the magnetic field energy despite the presence of diffusion.

Dynamo theory quickly became the accepted physical explanation for the presence of geophysical and astrophysical magnetic fields; see~\cite[Chapter 1]{childress_stretch_1995} for a historical sketch and extensive references. However, there remained a mathematical question at the heart of dynamo theory: can we exhibit a velocity field that demonstrable causes magnetic field growth (in the desired, physically correct, way). Despite considerable effort from a variety of direction, this question has largely eluded a clear resolution. Notably, a precise form of this question was given as one of Arnold's problems~\cite[Problem 1994-28]{arnold_arnolds_2005}. A discussion of the literature and of the primary difficulties of fully resolving the problem is given in Subsection~\ref{sub:background}. Let us first state precisely our setting and our result.

\begin{definition}
    For any vector field $u\in W^{1,\infty}(\R_+ \times \T^3)$ with $\nabla \cdot u = 0,$ any diffusivity $\kappa \geq 0$, and times $0 \leq s \leq t$, let $\sol^{u,\kappa}_{s,t}$ denote the solution operator for the induction equation associated to $u$ with diffusivity $\kappa$, that is $\sol^{u,\kappa}_{s,t} b(s) = b(t)$, where $b$ solves
    \begin{equation}
    \label{eq.induction-equation}
    \partial_t b - \kappa \Delta b + u \cdot \nabla b - b \cdot \nabla u =0.
    \end{equation}
\end{definition}

\eqref{eq.induction-equation} is a commonly used effective equation for the evolution of the magnetic field for a conducting fluid relevant for large objects such as planets and stars. Physically, the fluid motion $u$ is coupled to the magnetic field, but here we treat the purely \textit{kinematic problem}, in which we are free to specify the velocity field arbitrarily. We note however that the velocity fields of interest are divergence-free $\nabla \cdot u =0$. Also, the magnetic field---by Maxwell's laws---is required to be divergence-free, $\nabla \cdot b =0.$ See~\cite[Chapter 1.2]{childress_stretch_1995} for a discussion of the physical validity of this setting.

\begin{definition}
    For any vector field $u\in W^{1,\infty}(\R_+ \times \T^3)$ with $\nabla \cdot u = 0$ and any diffusivity $\kappa \geq 0$, define the \textit{dynamo rate} of the pair $(u,\kappa)$ as
    \[\gamma(u,\kappa) := \sup_{b \in L^2(\T^3), \nabla \cdot b =0} \limsup_{t \to \infty} \frac{1}{t} \log \|\sol^{u,\kappa}_{0,t} b\|_{L^2(\T^3)}^2.\]
\end{definition}

We note that the dynamo rate gives the \textit{fastest} asymptotic exponential growth rate for any (divergence-free) initial data.

\begin{definition}
    A flow $u$ is then said to be a \textit{fast dynamo} if
\[\liminf_{\kappa \to 0} \gamma(u,\kappa) >0.\]    
\end{definition}

A fast dynamo is then a velocity field in which the fastest exponential growth rate as uniformly bounded away from $0$ for all sufficiently small initial data. It is the construction of a fast dynamo which is the primary goal of mathematical dynamo theory.

In this work, we do not succeed in building a full fast dynamo but instead build the weaker \textit{subsequentially fast dynamo}, which we define as follows.

\begin{definition}
    We say a flow  $u\in W^{1,\infty}(\R_+ \times \T^3)$ with $\nabla \cdot u = 0$ is a \textit{subsequentially fast dynamo} if
    \[\limsup_{\kappa \to 0} \gamma(u,\kappa) >0.\]
\end{definition}

We note the only difference is the exchange of a $\liminf$ and a $\limsup$. We call it a subsequentially fast dynamo because we get fast dynamo behavior along appropriate subsequences $\kappa_j \to 0$. What we prove is somewhat stronger than the existence of a subsequentially fast dynamo; the precise result is the following.

\begin{theorem}
    \label{thm.main-thm}
    There exists some $\kappa_0 >0$ such that for any countable collection of diffusivity $(\kappa_j)_{j=1}^\infty \subseteq [0,\kappa_0]$, there exists a flow $u : [0,\infty) \times \T^3 \to \R^3$ with $\nabla \cdot u =0$ that has uniform dynamo action for each diffusivity $\kappa_j$. 

    Specifically, there is an initial data $b_0 := \sin(x) \mathrm{e}_z$ such that if $b^{\kappa}(t,x)$ solves
    \[\begin{cases}
        \partial_t b^\kappa - \kappa \Delta b^\kappa + u \cdot \nabla b^\kappa - b^\kappa \cdot \nabla u = 0\\
        b^\kappa(0,\cdot) = b_0(\cdot),
    \end{cases}\]
    then for any $j \in \N$,
    \[\limsup_{t \to \infty}\max_{|k| =1} \frac{1}{t} \log |\hat b^{\kappa_j}(t,k)|^2 \geq \tfrac{1}{4}.\]
    Hence in particular
    \[\gamma(u,\kappa_j) \geq \limsup_{t \to \infty}\max_{|k| =1} \frac{1}{t} \log \| b^{\kappa_j}(t)\|_{L^2(\T^3)}^2 \geq \tfrac{1}{4}.\]
    Further the flow obeys uniform estimates
    \begin{equation}
    \label{eq.unif-estimates}
    \|\partial_t^\ell \partial_x^\alpha u\|_{L^\infty_{t,x}} \leq C(\ell,\alpha),\end{equation}
    where the constant doesn't depend on the sequence $\kappa_j$.
\end{theorem}

Taking $\{\kappa_j\} = \Q \cap [0,\kappa_0]$, we get a subsequentially fast dynamo. In fact, we get a dynamo that cause growth with some uniformly positive exponential rate on a dense set of diffusivities in an interval about $0$.

\begin{corollary}
    There exists a subsequentially fast dynamo $u \in C^\infty(\R_+ \times \T^3)$.
\end{corollary}

\subsection{Background}
\label{sub:background}

The problem we study here is most properly called the \textit{kinematic dynamo problem}, as we take the velocity field $u$ as fixed and consider only the kinematic evolution of the magnetic field. There is also the substantially more difficult \textit{magnetohydrodynamic dynamo problem}; for example one can take a randomly stirred fluid at fixed positive viscosity coupled to the magnetic field evolution through the magnetohydrodynamics equations and ask about presence (or absence) of macroscopic magnetic fields at long times. For a discussion of some aspects of the dynamical problem, see~\cite[Chapter 12]{childress_stretch_1995}. There are also some rigorous mathematical works touching on this subject~\cite{friedlander_dynamo_1991,gerard-varet_oscillating_2005,gerard-varet_shear_2007}, though we are far from a resolution to the dynamo problem in this setting.

Despite the relative simplicity of the kinematic dynamo problem, which reduces to the analysis of purely linear PDE, there are still myriad substantial difficulties. The first is the presence of obstructions to dynamo action in lower dimensional geometries, known as the ``antidynamo theorems''. The Zeldovich antidynamo theorem~\cite{zeldovich_magnetic_1956} states that no planar flow $u(x,y,z) = (u_x(x,y,z), u_y(x,y,z),0)$ can be a dynamo ($\gamma(u,\kappa)  \leq 0$ for all $\kappa > 0$) and the Cowling antidynamo theorem~\cite{cowling_magnetic_1933} states that no axisymmetric flow can cause exponential growth of any axisymmetric initial data. These theorems prevent us from studying a purely two-dimensional problem and force us to work in the more spacious three dimensions. 

The simple existence of any kinematic dynamo, that is a pair $(u,\kappa)$ where $\gamma(u,\kappa)>0$, is not a trivial exercise. When $\kappa=0$, the problem becomes substantially simpler and essentially reduces to proving the positivity of the top Lyapunov exponent for Lagrangian particle trajectories in the advecting flow $u$; this connection has been clearly elucidated in~\cite{zelati_three-dimensional_2024} in which the authors prove $\gamma(u,0) > 0$ for a randomized, time-dependent version of the ABC flow. For $\kappa=0$, the Zeldovich antidynamo theorem no longer applies and two-dimensional dynamo action is now possible: in fact it's quite common. In particular, if $\phi$ solves the transport equation
\[\partial_t \phi + u \cdot\nabla \phi =0,\]
then $b:= \nabla^\perp \phi$ solves the induction equation~\eqref{eq.induction-equation} with $\kappa=0.$ Note then that by interpolation
\[\|b_t\|_{L^2} = \|\phi_t\|_{H^1} \geq \|\phi_t\|_{L^2}^2 \|\phi_t\|_{H^{-1}}^{-1} =\|\phi_0\|_{L^2}^2 \|\phi_t\|_{H^{-1}}^{-1}.\]
Thus if $u$ is exponentially mixing, so that $\|\phi_t\|_{H^{-1}}$ decays exponentially quickly, we much have that $\|b_t\|_{L^2}$ grows exponentially quickly, that is $\gamma(u,0) >0$. As such, the substantial literature on expoenentially mixing flows~\cite{bedrossian_lagrangian_2022, bedrossian_almost-sure_2022, blumenthal_exponential_2023,cooperman_exponential_2024, navarro-fernandez_exponential_2025} provide a wealth of examples of $\kappa=0$ dynamos.

For $\kappa>0$, the problem becomes more complex. However, perturbative arguments allow one to show that not only do kinematic dynamos exist, they are in some sense generic~\cite{roberts_spatially_1970} (see~\cite{zelati_alpha-unstable_2025} for a more modern and precise treatment of the perturbative argument). After having established the existence of dynamos for $\kappa=0$ and $\kappa>0$, it is natural to turn our attention to the getting the correct behavior of rates for small $\kappa$; the proper behavior as $\kappa \to 0$ is essential to physics of astrophysical and geophysical magnetic fields, in which one should imagine that there is an extremely small, but importantly non-zero, magnetic diffusivity. Vainshtein and Zeldovich introduced the distinction between slow and fast dynamos~\cite{vainshtein_origin_1972}. As stated above, a flow $u$ is said to be a fast dynamo if
\[\liminf_{\kappa\to0} \gamma(u,\kappa) >0.\]
Generally, all flows that are not fast dynamos are called slow dynamos, but the flows that are often held up as examples of slow dynamos are flow for which
\[\limsup_{\kappa \to 0} \gamma(u,\kappa) \leq 0.\]
Thus, in the conventional usage, the subsequentially fast dynamo we construct here may be called a slow dynamo, but it is substantially closer to a fast dynamo than the typical slow dynamo.

The difficulty in constructing a fast dynamo lies in the $\kappa \to 0$ limit. The first thing to note is that the limit is very singular. In particular, since $\gamma(u,\kappa)$ is defined with a $t \to \infty$ limit, there is no reason to generically hope that $\gamma(u,\kappa)$ is continuous as $\kappa \to 0.$ In fact, in two-dimensions, we know by the above discussion that for any exponentially mixing flow $\gamma(u,\kappa) >0$ for $\kappa = 0$ and $\gamma(u,\kappa) \leq 0$ for $\kappa> 0$; thus $\gamma$ has a jump discontinuity as $\kappa=0$.

Speaking very broadly, we have two rather different approaches to understanding dynamo action for $\kappa=0$ and for $\kappa$ macroscopic. When $\kappa=0$, there is no ``small-scale cutoff'' and so dynamo action can be sustained by the development of progressively smaller filaments of magnetic field lines. This is what happens in the case of $\kappa=0$ dynamos in 2D generated by exponentially mixing flows: we have that $\|b_t\|_{L^2}$ grows exponentially, but exponential mixing also guarantees that $\|\Pi_{\leq N} b_t\|_{L^2} \to 0$ for any fixed $N$ where $\Pi_{\leq N} f$ is the projection of $f$ onto all Fourier modes with wavenumber at most $N$. That is, the dynamo action is happening \textit{exclusively on small scales.}

On the other hand, when $\kappa$ is macroscopic, we can hope to essentially perturb away from the Laplacian. While the Laplacian is purely diffusive--and so causes no growth---it is a very well behaved operator with discrete, not overly degenerate spectrum. Thus the perturbative arguments for dynamo action when $\kappa>0$ actually are using the dissipation of the Laplacian to prove the dynamo growth. Interpolating between these two arguments---and uniformly controlling the growth rate in the interpolation---is where the essentially difficulty lies.

The setting of primary physical interest for dynamo action is that of a compact spatial domain, or at least of a compactly support velocity field: the Earth's molten core takes up only a finite region of space. This compactness forces us to deal with the issues posed by taking the $\kappa \to 0$ limit. That is because necessarily, in a compact domain, dynamo action can only be supported by the folding of magnetic field lines on themselves (this is the origin of the ``stretch twist fold'' motto, and is well elucidated in~\cite[Chapters 1-2]{childress_stretch_1995}). This folding necessarily creates small scales. As such, as we decrease $\kappa$, the magnetic field $b_t$ will become ever more complex, involving ever smaller scales. To build a fast dynamo, one would have to be able to understand the long-time dynamics of this increasingly complex object, uniformly in $\kappa$.

This issue of small scale production and folding can be completely elided, even on the compact domain of the torus, by considering dynamo action as induced by maps, such as Arnold's cat map. Since this avoids the main difficulty of this problem and is largely unphysical, we won't discuss it further (see however~\cite[Chapter 3]{childress_stretch_1995} for a thorough treatment). 

The setting of $\R^3$ is one of intermediate difficulty, in which the fast dynamo problem remains highly difficult yet some of the issues present in the compact setting can be sidestepped. The impressive recent work~\cite{zelati_alpha-unstable_2025} provides the construction of a fast dynamo on $\R^3$, positively resolving the fast dynamo problem in that setting. Their construction is of an autonomous vector field $u \in W^{1,\infty}(\R^3)$ such that for each diffusivity $\kappa \in (0,1)$, there exists a data $b_0$ such that if $b_t$ is the solution to~\eqref{eq.induction-equation}, then $\|b_t\|_{L^2(\R^3)} \geq C^{-1} e^{C^{-1} t} \|b_0\|_{L^2(\R^3)}$ for some universal constant $C.$ As such, their construction has the very nice property that the $\limsup$ in the definition of $\gamma(u,\kappa)$ could be replaced with a $\liminf$ in their setting. It is however still an interesting open question as to whether one can construct a \textit{smooth} fast dynamo on $\R^3$, e.g. where $u \in L^\infty_t C^k_x$ for all $k \in \N$. 

Let us also note for completeness that there are two other known examples of fast dynamos on $\R^3$, one given by linear advecting flows $u(x) = A x$~\cite{zeldovich_kinematic_1984} and the other the \textit{Ponomarenko dynamo}~\cite{gilbert_fast_1988}. However, the former is unbounded and the latter discontinuous, so neither is spatially Lipschitz, which is (seemingly) a natural minimal requirement on the velocity field.

Finally, the dynamo problem for a random velocity field---a particularly relevant example being the velocity field given by solving a randomly forced fluid equation---is an example of the more general problem in giving lower bounds for the top Lyapunov exponent for an infinite dimensional linear cocycle. Some aspects of this problem are studied in~\cite{hairer_lower_2024}. While the results of that work are far from establishing a positive top Lyapunov exponent (in their settings they lower bound the top Lyapunov exponent by $-\kappa^{-\alpha}$ for some $\alpha>0$, which goes to $-\infty$ as $\kappa \to 0$), it seems that the ideas present in that work may be useful in further investigations into fast dynamo action.

\subsection{Overview of the argument}

\label{sub:overview}

Our construction provides the first example of a velocity field on $\T^3$ that exhibits dynamo action in arbitrary neighborhoods of $\kappa =0$, that is the first subsequentially fast dynamo on $\T^3$. It also has the advantage being smooth and following a relatively simple construction and analysis.

The construction has the spirit of a control problem, in which we inspect the current state of the magnetic field $b_t$ in order to choose the velocity field $u$ on the next time interval $[t,t+2].$ Propositions~\ref{prop.exists-growing} and~\ref{prop.is-transitive} are the results that give us a rich enough class of controls. What they allow is for any $\kappa \in [0,\kappa_0]$ and initial data $b_0$ such that $b_0 \ne 0, \nabla \cdot b_0 = 0,$ and $\int b_0(x)\,dx =0$, we can construct a velocity field $u$ (which obeys uniform regularity estimates as in~\eqref{eq.unif-estimates}) such that if $b_t$ solves~\eqref{eq.induction-equation}, we have that $\|b_t\|_{L^2}$ grows at a uniform exponential rate $\gamma$.

That control argument only \textit{a priori} gives us exponential growth at \textit{a single diffusivity}. We then need to upgrade to exponential growth along our sequence of diffusivities $\kappa_j$. The idea is to simply cause growth on each diffusivity one after the other, visiting each diffusivity infinitely many times. That is, cause the solution associated to the diffusivity $\kappa_1$ to grow for a bit of time $[0,t_1]$, then look at what the solution for $\kappa_2$ looks like at time $t_1$, use that to construct the vector field $u$ on $[t_1,t_2]$ so that the solution associated to $\kappa_2$, $b^{\kappa_2}_t$,is growing at the uniform exponential rate. We then wait until $\frac{1}{t} \log \|b^{\kappa_2}_t\|_{L^2} \geq \frac{1}{2} \gamma$. Since we know we can make $b^{\kappa_2}_t$ grow with exponential rate $\gamma$, no matter how small $b^{\kappa_2}_{t_1}$ is, we know it will take some finite amount of time until $\frac{1}{t} \log \|b^{\kappa_2}_t\|_{L^2} \geq \frac{1}{2} \gamma.$ We call that time $t_2.$ Then we turn our attention back to $\kappa_1$. It may have shrunk on $[t_1,t_2]$, but using the same idea, we can make it some that on the time intreval $[t_2,t_3]$ we have growth so that for the solution associated to $\kappa^1$, $b^{\kappa_1}_t$, we have that $\frac{1}{t_3} \log \|b^{\kappa_1}_{t_3}\|_{L^2} \geq \frac{1}{2} \gamma.$ Proceeding in this way, we cause growth for $\kappa_2, \kappa_3, \kappa_1,\kappa_2,...$. In this way, for each diffusivity $\kappa_j$ we will have that for the associated solution $b^{\kappa_j}_t$, $\frac{1}{t} \log \|b^{\kappa_j}_t\|_{L^2} \geq \frac{1}{2} \gamma$ infinitely often; in particular we get that $\gamma(u,\kappa_j) \geq \frac{1}{2} \gamma$. 

We note that this construction only has $\|b^{\kappa_j}_t\|_{L^2} \approx e^{\gamma t/2}$ on a lacunary sequence of times, in contrast to the construction of~\cite{zelati_alpha-unstable_2025}. A consequence of this is that there is very little hope in getting growth for $\kappa \in [0,\kappa_0]$ such that $\kappa \ne \kappa_j$ by a simple approximation argument using density of the $\kappa_j$ in $[0,\kappa_0]$. It seems unlikely that the constructed velocity field could ever be proven to be a fast dynamo. As such, constructing a true fast dynamo would require a fairly different approach.

In the above sketch of the argument, we said we can cause exponential growth of the magnetic field solution associated to the diffusivity \textit{no matter how small}. This is only true if the solution is non-zero: once the solution becomes identically zero, it can never become anything other than zero. Fortunately, unique continuation results prevent the solution to the induction equation~\eqref{eq.induction-equation} from becoming $0$ in finite time. The case of the scalar advection diffusion equation is much better studied; in that case one can show unique continuation under very mild assumptions~\cite{poon_unique_1996}. We give a unique continuation result in Proposition~\ref{prop.doesnt-vanish} that suffices for our purposes. Proposition~\ref{prop.doesnt-vanish} is effectively a version of the well-known double exponential lower bound on the advection diffusion equation, e.g.\ as shown in~\cite[2.3.1]{miles_diffusion-limited_2018}.

Let us finally remark on how we prove the ``controllability'' results of Propositions~\ref{prop.exists-growing} and~\ref{prop.is-transitive}. Proposition~\ref{prop.is-transitive} is essentially a technical result that says that if $b_0 \ne 0$, then we can choose a flow $u$ such that at some future time $t$, $b_t$ has a nonzero mass on a Fourier mode of unit magnitude. This result follows essentially by a direct computation of derivatives, using the Fourier transform of the equation, allowing us to choose a velocity field that directly ``couples'' the mode of $b_0$ that has initial Fourier mass and some mode of unit magnitude that we want to put Fourier mass on.

Proposition~\ref{prop.exists-growing} is somewhat more interesting. The basic idea is given in Corollary~\ref{cor.averaged-dynamics}, which gives that the evolution---averaged over uniform translations of the advecting flow---is diagonal in Fourier space. That is, for any flow $u$, if $b^y_1 := \sol^{\tau_y u,\kappa}_{0,1} b_0$, then 
\begin{equation}
\label{eq.integrated-evolution-intro}
    \int \hat b^y_1(k)\,dy = A^k \hat b_0(k)
\end{equation}
for some explicitly computable matrix $A^k$. In order to prove Proposition~\ref{prop.exists-growing}, we have by hypothesis that for some $|k| =1$, that $\hat b_0(k) \ne 0$, and so we show that we can construct some flow $u \in C_c^\infty((0,2) \times \T^3)$ such that the associated matrix $A_k$, we have that $|A_k^n \hat b_0(k)| \geq e^n$ for $n$ sufficiently large. This then allows us to choose a sequence of translations of $u$ that we can glue together to make an advecting flow that causes $\hat b_n(k)$ to grow exponentially. The advecting flows we construct are similar those considered in~\cite[Chapter 11.4]{childress_stretch_1995}.

This identity~\eqref{eq.integrated-evolution-intro} is clearest from a slightly different perspective: that of random \textit{renewing flows}. We consider our advecting flow $u$ to be random, taken to be iid on unit time intervals $[n,n+1]$, and on each unit time interval $u(t,x) = u_0(t-n, x+Y_n)$ where $Y_n$ is an iid uniform translation. Then if we consider the (random) magnetic field given as $b_n(x) := \sol^{u,\kappa}_{0,n} b_0(x)$, the relation~\eqref{eq.integrated-evolution-intro} gives that the one-point correlation function of $b_n(x)$ has a particular simple evolution in Fourier space: $\E \hat b_{n+1}(k) = A_k \E \hat b_n(k)$. This fact has been well exploited in the physics and applied mathematics literatures, see~\cite[Chapter 11.1, 11.4]{childress_stretch_1995} for historical references. Renewing flows as particularly tractable random models of advecting flows has seen extensive usage in the passive scalar chaos literature---motivated by the mixing example of Pierrehumbert~\cite{pierrehumbert_tracer_1994}---in which they can be shown to almost surely have positive top Lyapunov exponents, exponentially mix passive scalars, and exponentially enhance passive scalar dissipation~\cite{blumenthal_exponential_2023,cooperman_harris_2024}. One can also view transport noise (also referred to as the Kraichnan model in the fluid mechanics setting, following~\cite{kraichnanSmallScaleStructure1968}) as a limiting case of renewing flows under taking the ``renewing time'' to $0$ and properly scaling the magnitude of the flow. In the study of dynamos, transport noise is often called the Kraichnan-Kazantsev model following~\cite{kraichnan_convection_1974,kazantsev_enhancement_1968}. 

In the mathematical study of transport noise dynamics---as well as in the passive scalar chaos work on renewing flows---extensive use of the two-point correlation function (in our setting this would be $\E b(t,x) \otimes b(t,y)$) is made. The two-point correlation function also has somewhat simple dynamics, but it is infinitely more complicated than a simple matrix multiplication on $\C^3$. However, the works using the two-point correlation function---such as to prove exponential mixing~\cite{gess_stabilization_2021,luo_elementary_2024}, anomalous dissipation~\cite{jan_integration_2002,jan_flows_2004,rowan_anomalous_2024}, or anomalous regularization~\cite{coghi_existence_2024,galeati_anomalous_2024,bagnara_anomalous_2024}---prove some object (namely an $H^s$ norm of the solution for some $s \in (-1/2,1/2)$ depending on the problem) is suitably \textit{small}. As such they need a coercive quantity that they are taking an expectation of: if the expected $H^s$ norm is small, then the $H^s$ norm has to be typically small. The one point correlation function is never coercive: it can vanish identically despite having highly nontrivial dynamics of the underlying field $b_n(x)$. However, we just want to construct a solution where $\|b_n\|_{L^2_x}$ is large, for which it suffices to have that $\E \hat b_n(k)$ is large: if the expectation is large, then the random variable must be large at least on some event. Thus in this setting, it suffices to work with the much simpler but non-coercive one-point correlation function.

\subsection{Acknowledgments}

The author was partially supported by the NSF Collaborative Research Grant DMS-2307681 and the Simons Foundation through the Simons Investigators program.

\section{Proof of Theorem~\ref{thm.main-thm}}

Let us first introduce some notation for the natural space that our magnetic field $b$ takes values in. We note that the solution operator is a well-defined map $\sol^{u,\kappa}_{s,t} : L^2_{\mathrm{sol},0}(\T^3) \to L^2_{\mathrm{sol},0}(\T^3),$ that is it takes divergence-free, mean-zero magnetic fields to divergence-free, mean-zero magnetic fields.

\begin{definition}
    Denote
    \[L^2_{\mathrm{sol},0}(\T^3) := \{ b \in (L^2(\T^3))^3 : \nabla \cdot b =0, \int_{\T^3} b(x)\,dx =0\}.\]
\end{definition}

Our first proposition is the main proposition we use to control the magnetic field and cause growth. It gives that, provided there is some initial mass of $b_0$ on the Fourier modes of unit magnitude, then we can choose the velocity field $u$ to cause exponential growth at the desired rate while obeying uniform estimates.

\begin{proposition}
    \label{prop.exists-growing}
    There exists some $\kappa_0 >0$ such that for any $\kappa \in [0,\kappa_0]$, and any $b_0 \in L^2_{\mathrm{sol},0}(\T^3)$ such that $\hat b_0(k) \ne 0$ for some $k \in \Z^3$ with $|k| =1$, there exists a flow $u : [0,\infty) \times \T^3 \to \R^3$ with $\nabla \cdot u =0$ and a time $2n \in 2\N$ such that,
    \[|\widehat{\sol_{0,2n}^{u,\kappa} b_0}(k)| \geq e^{n}.\]
    Further, $u$ can be taken so that for any $\ell,\alpha$,
    \[\|\partial_t^\ell \partial_x^\alpha u\|_{L^\infty_{t,x}} \leq C(\ell,\alpha),\]
    where the constant doesn't depend on $b_0$ or $\kappa$. Additionally, we can take $u$ so that $u \in C_c^\infty((0,2n) \times \T^3).$
\end{proposition}

The second proposition is a unique continuation result. It guarantees if $b_0$ is non-trivial (and sufficiently smooth), then $b_t$ is non-zero for all finite $t$. 

\begin{proposition}
    \label{prop.doesnt-vanish}
    Let $b \in L^2_{\mathrm{sol},0}(\T^3)$ solve
    \[\partial_t b - \kappa\Delta b + u \cdot \nabla b - b \cdot \nabla u = 0.\]
    Then
    \[\|b\|_{L^2_x}^2(t) \geq \exp\Bigg({-2}\bigg(\int_0^t \kappa \exp\Big(C\int_0^s \|\nabla^2 u\|_{L^\infty}(r)\,dr \Big)\frac{\|\nabla b(0,\cdot)\|_{L^2_x}^2}{\|b(0,\cdot)\|_{L^2_x}^2}
    + \|\nabla u\|_{L^\infty_x}(s)\,ds\bigg)\Bigg) \|b(0,\cdot)\|_{L^2_x}.\]
    In particular, if $\|b(0,\cdot)\|_{L^2_x} \ne 0, \|b(0,\cdot)\|_{H^1_x} < \infty,$ and $\nabla^2 u \in L^1([0,T], L^\infty(\T^3)),$ then $b(T,\cdot) \ne 0$.
\end{proposition}

Our final proposition is another controllability result, which allows us to move some mass onto Fourier modes of unit magnitude provided $b_0$ is nontrivial.

\begin{proposition}
    \label{prop.is-transitive}
     Let $b_0 \in L^2_{\mathrm{sol},0}(\T^3)$ such that $b_0 \ne 0$ and $\kappa \geq 0$. Then there exists a time $T>0,$ a Fourier mode $k \in \Z^3$ such that $|k| = 1$, and a flow $u : [0,\infty) \times \T^3 \to \R^3$ with $\nabla \cdot  u=0$ such that if $b(t,x)$ solves 
    \[\begin{cases}
        \partial_t b - \kappa \Delta b + u \cdot \nabla b - b \cdot \nabla u = 0\\
        b(0,\cdot) = b_0(\cdot),
    \end{cases}\]
    then for any $t \geq T$, $\hat b(t,k) \ne 0$. Further, the flow can be taken so that $u \in C_c^\infty((0,T) \times \T^3)$ and so that for any $\ell,\alpha$,
    \[\|\partial_t^\ell \partial_x^\alpha u\|_{L^\infty_{t,x}} \leq C(\ell,\alpha),\]
    where the constant doesn't depend on $b_0$ or $\kappa$.
\end{proposition}

As a consequence of these propositions, we can prove the main theorem according to the scheme sketched in Subsection~\ref{sub:overview}.

\begin{proof}[Proof of Theorem~\ref{thm.main-thm}]
    Place the diffusivities $\kappa_j$ in the order 
    \[(\tilde \kappa_1, \tilde \kappa_1,\tilde \kappa_2, \tilde \kappa_1, \tilde \kappa_2,\tilde \kappa_3,...)= (\kappa_1,\kappa_1,\kappa_2,\kappa_1,\kappa_2,\kappa_3,....),\] so that each diffusivity $\kappa_j$ is visited infinitely many times in the sequence $\tilde \kappa_j$.

    Let $b_0 := \sin(x)\mathrm{e}_z$ and let $b^\kappa$ be the solution to~\eqref{eq.induction-equation} with $\kappa$ diffusivity and $b^\kappa(0,\cdot) = b_0(\cdot)$ as in the statement of the theorem. 
    
    The flow $u$ will be constructed inductively as follows. We inductively construct times $t_n$ and specify the flow $u$ on $[t_{n-1}, t_n]$ in such a way that
    \begin{enumerate}
        \item  \label{item.growth}\[\sup_{t \in [t_{n-1},t_n]} \max_{|k|=1}\frac{1}{t} \log |\hat b^{\tilde \kappa_n}(t,k)|^2 \geq \tfrac{1}{4}.\]
        \item $t_n - t_{n-1} \geq 1$.
        \label{item.enough-time}
        \item \label{item.unif-est} $u$ obeys the uniform estimates~\eqref{eq.unif-estimates} on $[t_{n-1}, t_n]$.
        \item \label{item.support-cond}For any $n$, $u|_{[t_{n-1}, t_n]} \in C_c^\infty((t_{n-1},t_n) \times \T^3)$.
    \end{enumerate}
   One can readily verify that completing such a construction suffices to prove the theorem, where the final item is to ensure the distinct intervals of the flow can be smoothly patched together.

The proof of the base case of the induction is just a simpler version of the inductive step, so let us just show the inductive step.

   We first take $u|_{[t_{n-1}, t_{n-1}+1]}  =0$, thus Item~\ref{item.enough-time} is satisfied. Next, by Proposition~\ref{prop.doesnt-vanish}, we have that $b^{\tilde \kappa_n}(t_{n-1} +1) \in L^2_{\mathrm{sol},0}(\T^3)$ and is nonzero. Thus we can apply Proposition~\ref{prop.is-transitive} to specify $u|_{[t_{n-1} +1,t_{n-1}+1+T]}$ so that $\hat b^{\tilde \kappa_n}(t_{n-1}+1+T,k) \ne 0$ for some $|k| =1$. We can then apply Proposition~\ref{prop.exists-growing} to choose $R \in 2\N$ large enough and specify $u|_{[t_{n-1} +1 + T, t_{n-1} + 1+ T +R]}$ in such a way that for $t_n := t_{n-1} + 1 + T +R$,
    \[\max_{|k| =1} \frac{1}{t_n} \log |\hat b^{\tilde \kappa_n}(t_n,k)|^2 \geq \tfrac{1}{4}.\]
    This then gives Item~\ref{item.growth}. Items~\ref{item.unif-est} and~\ref{item.support-cond} follow directly from the uniform estimates and conditions on the time support of $u$ in Proposition~\ref{prop.exists-growing} and Proposition~\ref{prop.is-transitive}. Thus the induction concludes.
\end{proof}

\section{Proof of Proposition~\ref{prop.exists-growing}}
\begin{definition}
    For $y \in \T^3$, let $\tau_y$ denote the translation operator by $y$,
    \[\tau_y u(t,x) := u(t,x-y).\]
\end{definition}

\begin{definition}
    Fix $u \in L^\infty_tW^{1,\infty}_x$ with $\nabla \cdot u  =0 , \kappa \geq 0$, and $0 \leq s < t <\infty.$ Then for $k,j \in \Z^3,$ we define the Fourier matrix element $\hat \sol^{u,\kappa}_{s,t}(k,j) \in \R^{3 \times 3}$ of the solution operator $\sol_{s,t}^{u,\kappa}$ so that for any $v \in \C^3,$
    \[\hat \sol^{u,\kappa}_{s,t}(k,j) v:= \int e^{- 2\pi i k \cdot x} \sol^{u,\kappa}_{s,t} ( v e^{2\pi i j \cdot z})\,dx.\]
\end{definition}

Taking the usual convention for the Fourier transform,
\[b(x) =: \sum_{j \in \Z^3} e^{2\pi i j \cdot x} \hat b(j),\]
we note that the above definition is given precisely so that
\begin{equation}
\label{eq.matrix-element-mult}
\widehat{\sol^{u,\kappa}_{s,t} b}(k) = \sum_{j \in \Z^d} \hat \sol^{u,\kappa}_{s,t}(k,j) \hat b(j).
\end{equation}
The primary fact we will be exploiting is the following simple observation on how Fourier matrix elements transform under translations of the advecting flow.

\begin{lemma}
    \label{lem.translation-fourier}
    Fix $u \in L^\infty_tW^{1,\infty}_x$ with $\nabla \cdot u  =0 , \kappa \geq 0$, and $0 \leq s < t <\infty.$ The for any $k,j \in \Z^3$ and $y \in \T^3$, we 
    \[ \hat \sol^{\tau_y u,\kappa}_{s,t}(k,j) = e^{2\pi i (j-k) \cdot y} \hat \sol^{u,\kappa}_{s,t}(k,j).\]
\end{lemma}

\begin{proof}
    Note that for any $b \in L^2_{\mathrm{sol},0}$, we have that
    \[\sol_{s,t}^{\tau_y u,\kappa} \tau_y b = \tau_y \sol_{s,t}^{u,\kappa} b.\]
    Therefore,
    \begin{align*}
        \hat \sol^{\tau_y u,\kappa}_{s,t}(k,j)v 
      &= \int e^{- 2\pi i k \cdot x} \tau_{y}\sol^{u,\kappa}_{s,t} ( v \tau_{-y}e^{2\pi i j \cdot z})\,dx 
    \\&=  \int e^{- 2\pi i k \cdot x} \Big(\sol^{u,\kappa}_{s,t} ( v e^{2\pi i j \cdot (z+y)})\Big)(x-y)\,dx 
        \\&=  e^{2\pi i (j-k)\cdot y} \int e^{- 2\pi i k \cdot (x-y)} \Big(\sol^{u,\kappa}_{s,t} ( v e^{2\pi i j \cdot z})\Big)(x-y)\,dx
        \\&=  e^{2\pi i (j-k)\cdot y} \int e^{- 2\pi i k \cdot x} \sol^{u,\kappa}_{s,t} ( v e^{2\pi i j \cdot z})\,dx 
        \\&= e^{2\pi i (j-k) \cdot y} \hat T^{u,\kappa}_{s,t}(k,j) v,
     \end{align*}
     as desired.
\end{proof}

From the above computations, we can understand the \textit{averaged} evolution of a single Fourier in terms of simple matrix products on $\C^3$.

\begin{corollary}
    \label{cor.averaged-dynamics}
    Fix $u \in L^\infty_tW^{1,\infty}_x$ with $\nabla \cdot u  =0 , \kappa \geq 0$, and $0 \leq s < t <\infty.$ Then for any $b \in L^2_{\mathrm{sol},0}$ and any $k\in \Z^d$,
    \[\int \widehat{\sol_{s,t}^{\tau_y u,\kappa} b}(k)\,dy = \hat \sol_{s,t}^{u,\kappa}(k,k) \hat b(k).\]
\end{corollary}

\begin{proof}
    Direct from~\eqref{eq.matrix-element-mult} and Lemma~\ref{lem.translation-fourier}.
\end{proof}

\begin{definition}
    Through we denote the standard basis vectors on $\R^3$
    \[\mathrm{e}_x := (1,0,0),\quad \mathrm{e}_y := (0,1,0),\quad \mathrm{e}_z := (0,0,1).\]
\end{definition}

We now state a proposition that asserts the existence of flows $u_1,u_2 \in C_c^\infty((0,2) \times \T^3)$ for which we have properties on the Fourier matrix elements $\hat \sol^{u_1,0}_{0,2}(\mathrm{e}_z, \mathrm{e}_z),\hat \sol^{u_2,0}_{0,2}(\mathrm{e}_z, \mathrm{e}_z)$ sufficient to cause growth in the averaged dynamics given by Corollary~\ref{cor.averaged-dynamics} for an arbitrary nontrivial $\hat b_0(\mathrm{e}_z) \in \C^3$. Note that since $b_0$ is divergence-free, we have that $\mathrm{e}_z \cdot \hat b_0(\mathrm{e}_z) =0$; this is why we consider $v \in \C^3, v\ne 0$ with $\mathrm{e}_z \cdot v =0$ in Proposition~\ref{prop.is-controls} and Corollary~\ref{cor.diffusive-controls} below.

\begin{proposition}
\label{prop.is-controls}
    There exists $u_1,u_2 \in C_c^{\infty}((0,2) \times \T^3)$ with $\nabla \cdot u_1 = \nabla \cdot u_2 =0$ with the following properties. Let
    \[A_1 := \hat \sol^{u_1,0}_{0,2}(\mathrm{e}_z,\mathrm{e}_z) \quad \text{and} \quad A_2 := \hat \sol^{u_2,0}_{0,2}(\mathrm{e}_z,\mathrm{e}_z).\]
    Then $A_1,A_2$ each have only simple eigenvalues. Let $\xi_{1,1},\xi_{1,2},\xi_{1,3}$ and $\xi_{2,1},\xi_{2,2},\xi_{2,3}$ be the eigenvector sets of $A_1,A_2$ respectively. Let $\lambda_1, \lambda_2$ the eigenvalues associated to $\xi_{1,1}$ and $\xi_{2,1}$ respectively. Then we have that
    \begin{enumerate}
        \item 
        \label{item.evals-big}
        $|\lambda_1|, |\lambda_2|> e$.
        \item 
        \label{item.span-cond}
        For all $v \in \C^3, v \ne 0$ with $\mathrm{e}_z \cdot v =0,$
        \[v \not \in \mathrm{span}\{\xi_{1,2}, \xi_{1,3}\} \quad \text{or} \quad v \not \in \mathrm{span}\{\xi_{2,2}, \xi_{2,3}\}.\]
    \end{enumerate}
\end{proposition}

We defer the mostly computational proof to Subsection~\ref{sub:construction}. Let us see how to conclude Proposition~\ref{prop.exists-growing} from the existence of the ``controls'' as in Proposition~\ref{prop.is-controls}. Noting that Proposition~\ref{prop.is-controls} is only about the non-diffusive problem, we first have to lift the properties of Proposition~\ref{prop.is-controls} to the diffusive case $\kappa>0$. This is provided by Corollary~\ref{cor.diffusive-controls}. Since we are only interested in a single matrix element of the finite time evolution, the diffusion acts as a regular perturbation. Note that this is a key point of the analysis: reducing the problem to a finite time problem makes the Laplacian a regular perturbation in contrast to the highly singular behavior at infinite time.

\begin{corollary}
\label{cor.diffusive-controls}
    There exists $u_1,u_2 \in C_c^{\infty}((0,2) \times \T^3)$ with $\nabla \cdot u_1 = \nabla \cdot u_2 =0$ as well as $\kappa_0 >0$ with the following properties. For any $\kappa \geq 0$, let
    \[A_1^\kappa := \hat \sol^{u_1,\kappa}_{0,2}(\mathrm{e}_z,\mathrm{e}_z) \quad \text{and} \quad A_2^\kappa := \hat \sol^{u_2,\kappa}_{0,2}(\mathrm{e}_z,\mathrm{e}_z).\]
    Then for all $v \in \C^3, v \ne 0$ such that $\mathrm{e}_z \cdot v =0$ and all $\kappa \in [0,\kappa_0]$, 
    \[\liminf_{n \to \infty} \frac{1}{n}\log |(A_1^\kappa)^n v| > 1 \quad \text{or}\quad \liminf_{n \to \infty} \frac{1}{n}\log |(A_2^\kappa)^n v| > 1.\]
\end{corollary}

\begin{proof}
    Since $A_1^0,A_2^0$ have only simple eigenvalues, we know that $A_1^\kappa, A^\kappa_2$ do also for $\kappa\in[0,\kappa_0]$ for $\kappa_0>0$ chosen sufficiently small. Then $A_1^\kappa, A^\kappa_2$ are diagonalizable, so let $\xi^\kappa_{1,1},\xi^\kappa_{1,2},\xi^\kappa_{1,3}$ and $\xi^\kappa_{2,1},\xi^\kappa_{2,2},\xi^\kappa_{2,3}$ be their eigenvectors, varying continuously in $\kappa$, with $\xi^0_{1,1} = \xi_1$ and $\xi^0_{2,1}= \xi_2$. Let $\lambda_1^\kappa$ be the eigenvalue associated to $\xi^\kappa_{1,1}$ and $\lambda_2^\kappa$ the eigenvalue associated to $\xi^\kappa_{2,1}.$ Then, by continuity, for $\kappa \in [0,\kappa_0]$ and $\kappa_0$ chosen sufficiently small, $|\lambda_1^\kappa|, |\lambda_2^\kappa| > e.$ Then either
    \[\liminf_{n \to \infty} \frac{1}{n}\log |(A_1^\kappa)^n v| > 1 \quad \text{or}\quad \liminf_{n \to \infty} \frac{1}{n}\log |(A_2^\kappa)^n v| > 1,\]
    in which case we conclude, or 
    \[
        \label{eq.bad-span} v \in \mathrm{span}\{\xi^\kappa_{1,2}, \xi^{\kappa}_{1,3}\} \quad \text{and} \quad v \in \mathrm{span}\{\xi^\kappa_{2,2}, \xi^{\kappa}_{2,3}\}. 
    \]
    Thus we just need to prove
            \[v \not \in \mathrm{span}\{\xi^\kappa_{1,2}, \xi^\kappa_{1,3}\} \quad \text{or} \quad v \not \in \mathrm{span}\{\xi^\kappa_{2,2}, \xi^\kappa_{2,3}\}.\]
    This is then direct from Item~\ref{item.span-cond} of Proposition~\ref{prop.is-controls}, the continuity in $\kappa$, that the span condition is an open condition, and choosing $\kappa_0$ small enough.
\end{proof}

With sufficient understanding of the matrix elements at all diffusivities $\kappa \in [0,\kappa_0]$, we are now prepared to prove Proposition~\ref{prop.exists-growing}.

\begin{proof}[Proof of Proposition~\ref{prop.exists-growing}]
    By hypothesis, there exists $|k| = 1$ such that $\hat b_0(k) \ne 0$. Without loss of generality---since otherwise we can perform a discrete rotation---we can suppose that $\hat b_0(\mathrm{e}_z) \ne 0.$ We let $\kappa_0$ as in Corollary~\ref{cor.diffusive-controls} and let $\kappa \in [0,\kappa_0].$ Note that by the divergence-free condition on $b_0$, we have that $\mathrm{e}_z \cdot \hat b_0(\mathrm{e}_z) =0.$ By Corollary~\ref{cor.diffusive-controls}, let $A^\kappa_1, A^\kappa_2$ as in the statement of the Corollary~\ref{cor.diffusive-controls}, we thus have
    \[\liminf_{n \to \infty} \frac{1}{n}\log |(A_1^\kappa)^n \hat b_0(\mathrm{e}_z)| > 1 \quad \text{or}\quad \liminf_{n \to \infty} \frac{1}{n}\log |(A_2^\kappa)^n \hat b_0(\mathrm{e}_z)| > 1.\]
    Let us suppose without loss of generality that the inequality holds for $A_1^\kappa$. By Corollary~\ref{cor.averaged-dynamics} and the definition of $A_1^\kappa$, 
    \begin{equation}
    \label{eq.iterated-averaged-dynamics}
    \int e^{-2\pi i \mathrm{e}_z\cdot x}\Big(\sol^{\tau_{y_n} u_1, \kappa}_{0,2} \sol^{\tau_{y_{n-1}} u_1,\kappa}_{0,2} \cdots \sol^{\tau_{y_1} u_1,\kappa}_{0,2} b_0\Big)(x)\, dy_1 \cdots dy_n dx = (A_1^\kappa)^n \hat b_0(\mathrm{e}_z).
    \end{equation}
    For any sequence $y_j \in \T^3$, we define
    \[u_{(y_j)}(t,x) = \tau_{y_j} u_1(t-2(j-1), x) \quad \text{for } t \in [2(j-1),2j],\]
    and let 
    \[b_{(y_j)}(t,x) := \Big(\sol_{0,t}^{u_{(y_j)}, \kappa} b_0\Big)(x).\]
    
    Then by~\eqref{eq.iterated-averaged-dynamics},
    \[\int \widehat{b_{(y_j)}}(2n,\mathrm{e}_z)\,dy_1 \cdots dy_n= (A^\kappa_1)^n \hat b_0(\mathrm{e}_z).\]
    Since $\liminf_{n \to \infty} \frac{1}{n}\log |(A_1^\kappa)^n \hat b_0(\mathrm{e}_z)| > 1$ we can choose $n \in \N$ such that
    \[ \int |\widehat{b_{(y_j)}}(2n,\mathrm{e}_z)|\,dy_1 \cdots dy_n \geq |(A_1^\kappa)^n \hat b_0(\mathrm{e}_z)| \geq e^n.\]
    Using that $dy_j$ is a probability measure, we can then choose $(y_1,...,y_n)$ so that
    \[|\widehat{\sol^{u,\kappa}_{0,2n} b_0}(\mathrm{e}_z)| \geq e^n,\]
    with 
    \[u(t,x) = \begin{cases}
        u_{(y_1,...,y_n,0,0,...)}(t,x) & t \in [0,2n]\\ 0 & \text{otherwise}.
    \end{cases}\]
    We then conclude, noting that the regularity and support properties of $u$ are direct from the construction as translations of a smooth function $u_1 \in C_c^\infty((0,2) \times \T^3).$
\end{proof}

\subsection{Construction of the controls: proof of Proposition~\ref{prop.is-controls}}
\label{sub:construction}
In this subsection we always use the coordinates $(x,y,z) \in\T^3$ with $x,y,z \in \T$ and in particular never take $x \in \T^3.$ Instead we use $(x,y,z) = r \in \T^3$.

\begin{definition}
    Let $\phi \in C_c^\infty((0,1/2))$ be a nonnegative function such that $\int_0^{1/2} \phi(t)\,dt = 1$. Define the flows $U_\lambda,V_\lambda,W_\lambda$ as
\begin{align*}
U_\lambda(t,x,y,z) &= \lambda \phi(t)\cos(2\pi x) \mathrm{e}_y  + \frac{1}{4}\phi(t-1/2) \sin(2\pi x)\mathrm{e}_z\\
V_\lambda(t,x,y,z) &= \lambda \phi(t)\cos(2\pi y) \mathrm{e}_x  + \frac{1}{4}\phi(t-1/2) \sin(2\pi y)\mathrm{e}_z\\
W_\lambda(t,x,y,z) &= U_\lambda(t,x,y,z) + V_{-\lambda}(t-1,x,y,z).
\end{align*}
\end{definition}

We note the flow $W_\lambda$ is similar to Otani's fast dynamo candidate~\cite{otani_fast_1993} which is well discussed in~\cite[Chapter 2, Chapter 11.4]{childress_stretch_1995}. Our goal in this section is to compute explicitly the matrix elements of the solution operator associated to this flow. We start by computing those associated to $U_\lambda, V_\lambda$.

\begin{lemma}
\label{lem.matrix-comp}
\begin{equation}
\label{eq.explicit-matrices}
\hat\sol^{U_\lambda,0}_{0,1}(\mathrm{e}_z,\mathrm{e}_z) = \begin{pmatrix} \alpha & 0 & 0 \\ i \lambda \beta  & \alpha& 0 \\ 0 & 0 & \alpha\end{pmatrix}; \quad  \hat\sol^{V_\lambda,0}_{0,1}(\mathrm{e}_z,\mathrm{e}_z) = \begin{pmatrix} \alpha & i \lambda \beta & 0 \\ 0  & \alpha& 0 \\ 0 & 0 & \alpha\end{pmatrix};\end{equation}
where $\alpha \in (0,1), \beta>0$. In particular,
\[\alpha = J_0(\tfrac{\pi}{2});\quad \beta = 2\pi J_1(\tfrac{\pi}{2}),\]
where the $J_n$ are the Bessel functions of the first kind.
\end{lemma}

\begin{proof}
    Note that the result for $V_\lambda$ follows symmetrically from the result for $U_\lambda$, so let us focus on the latter. By changing time variables, it suffices to compute the dynamics under the flow given by $\tilde U_\lambda$, where
    \[\tilde U_\lambda(t,x,y,z) = \begin{cases} 2 \lambda\cos(2\pi x) \mathrm{e}_y & 0 \leq t <1/2\\ \frac{1}{2} \sin(2\pi x)\mathrm{e}_z & 1/2 \leq t \leq1. \end{cases}\]

    Writing out the equation with advecting flow $\tilde U_\lambda$ explicitly in coordinates, we get
    \[\begin{cases} \partial_t b + 2 \lambda \cos(2\pi x) \partial_y b + 4 \pi \lambda \sin(2\pi x) b_x\mathrm{e}_y =0 & t \in [0,1/2]  \\ \partial_t b + \tfrac{1}{2} \sin(2\pi x) \partial_z b - \pi \cos(2\pi x)  b_x \mathrm{e_z} =0 & t \in [1/2,1].\end{cases}\]
    We need to compute the evolution under this PDE of the initial data of the form $e^{2\pi i z} v$ for $v \in \C^3$.

    Note that in the case under consideration, $\partial_y b(0,r) =0$, and that differentiating the equation by $\partial_y$ we see this property is preserved by the flow. Thus $\partial_y b(t,r) =0$. Similarly, we have that $\partial_z b(0,r) = 2\pi i b(0,r)$, and one can verify (e.g.\ by taking the Fourier transform) that this property is preserved by the flow, so we have that $\partial_z b(t,r) = 2\pi i b(t,r).$ Thus the PDE, for the data under consideration, becomes
    \[\begin{cases} \partial_t b + 4 \pi \lambda \sin(2\pi x) b_x\mathrm{e}_y =0 & t \in [0,1/2]  \\ \partial_t b + i\pi \sin(2\pi x) b - \pi \cos(2\pi x)  b_x \mathrm{e_z} =0 & t \in [1/2,1].\end{cases}\]
    Then we see that 
    \[\begin{pmatrix} b_x(\tfrac{1}{2}, r) \\ b_y(\tfrac{1}{2}, r) \\ b_z(\tfrac{1}{2}, r) \end{pmatrix} = \begin{pmatrix} b_x(0, r) \\  -2\pi \lambda \sin(2\pi x)b_x(0,r) + b_y(0, r) \\ b_z(0, r) \end{pmatrix}\]
    and
    \[\begin{pmatrix} b_x(1, r) \\ b_y(1, r) \\ b_z(1, r) \end{pmatrix}  =\begin{pmatrix} e^{-i \pi \sin(2\pi x) /2} b_x(\tfrac{1}{2}, r) \\  e^{-i \pi \sin(2\pi x) /2}b_y(\tfrac{1}{2}, r) \\  e^{-i \pi \sin(2\pi x) /2}b_z(\tfrac{1}{2}, r) + \pi \cos(2\pi x) b_x(\tfrac{1}{2},r) \int_0^{1/2} e^{-i \pi \sin(2\pi x) t}\,dt\end{pmatrix}.\]
    Thus combining the computations,
    \[\begin{pmatrix} b_x(1, r) \\ b_y(1, r) \\ b_z(1, r) \end{pmatrix}  =\begin{pmatrix} e^{-i \pi \sin(2\pi x) /2} b_x(0, r) \\  {-2\pi} \lambda \sin(2\pi x) e^{-i \pi \sin(2\pi x) /2}b_x(0,r) + e^{-i \pi \sin(2\pi x) /2} b_y(0, r) \\  e^{-i \pi \sin(2\pi x) /2}b_z(0, r) + \pi \cos(2\pi x) b_x(0,r) \int_0^{1/2} e^{-i \pi \sin(2\pi x) t}\,dt\end{pmatrix}.\]
    Thus if $b(0,r) = e^{2\pi i z} v$ with $v \in \C^3$, we have that
    \[\begin{pmatrix} \hat b_x(1, \mathrm{e}_z) \\ \hat b_y(1, \mathrm{e}_z) \\ \hat b_z(1, \mathrm{e}_z) \end{pmatrix}  =\begin{pmatrix} v_x\int e^{-i \pi \sin(2\pi x) /2}\,dx \\  {-2\pi} \lambda v_x\int \sin(2\pi x) e^{-i \pi \sin(2\pi x) /2}\,dx  + v_y \int e^{-i \pi \sin(2\pi x) /2}\,dx\\  v_z\int e^{-i \pi \sin(2\pi x) /2}\,dx  + \pi v_x  \int_0^{1/2} \int \cos(2\pi x) e^{-i \pi \sin(2\pi x) t}\,dxdt\end{pmatrix}.\]
    We recall the Hansen-Bessel formula,
    \[J_n(z) = (-1)^n \int e^{i z \sin(2\pi x) + 2\pi i nx}\,dx,\]
    where $J_n(z)$ is the $n$-th Bessel function of the first kind. Then we see that 
    \[\int e^{-i \pi \sin(2\pi x)/2}\,dx = J_0(-\tfrac{\pi}{2}) = J_0(\tfrac{\pi}{2})\]
    using that $J_n(z)$ is even for even $n$. We also have that
    \begin{align*}\int \sin(2\pi x) e^{-i \pi \sin(2\pi x)/2}\,dx &= \frac{1}{2i}\int e^{2\pi i x} e^{-i \pi \sin(2\pi x)/2} -  e^{-2\pi i x} e^{-i \pi \sin(2\pi x)/2}\,dx
    \\&= \frac{1}{2i}\int e^{2\pi i x} e^{-i \pi \sin(2\pi x)/2} -  e^{2\pi i x} e^{i \pi \sin(2\pi x)/2}\,dx
    \\&= -\frac{1}{2i} \big(J_1(-\tfrac{\pi}{2}) - J_1(\tfrac{\pi}{2})\big) = -i J_1(\tfrac{\pi}{2}),
    \end{align*}
    using that $J_n(z)$ is odd for odd $n$. Finally, we compute
    \begin{align*}
        \int \cos(2\pi x) e^{-i \pi \sin(2\pi x) t}\,dx &= \frac{1}{2}\int e^{2\pi i x} e^{-i \pi \sin(2\pi x) t}+e^{-2\pi i x} e^{-i \pi \sin(2\pi x) t}\,dx
        \\&= -\frac{1}{2} \big(J_1(-\pi t) + J_1(\pi t)\big) = 0,
    \end{align*}
    where again use that $J_1(z)$ is odd. Putting it together, we conclude~\eqref{eq.explicit-matrices}. That $\alpha \in (0,1), \beta>0$ follow from explicit properties of the Bessel functions.
\end{proof}

We now deduce the matrix element associated $W_\lambda$ from the above computation.

\begin{lemma}
    \[\hat\sol^{W_\lambda,0}_{0,2}(\mathrm{e}_z,\mathrm{e}_z) = \begin{pmatrix} \alpha^2 + \lambda^2 \beta^2 & -i\lambda \alpha \beta & 0 \\ i \lambda \alpha \beta  & \alpha^2& 0 \\ 0 & 0 & \alpha^2\end{pmatrix},\]
    with $\alpha, \beta$ as in Lemma~\ref{lem.matrix-comp}.
\end{lemma}

\begin{proof}
    Note that
    \[\hat\sol^{W_\lambda,0}_{0,2}(\mathrm{e}_z,\mathrm{e}_z) = \sum_{k \in \Z^3}  \hat\sol^{V_{-\lambda},0}_{0,1}(\mathrm{e}_z,k)   \hat\sol^{U_\lambda,0}_{0,1}(k,\mathrm{e}_z).\]
    Note also that
    \[\begin{pmatrix} \alpha^2 + \lambda^2 \beta^2 & -i\lambda \alpha \beta & 0 \\ i \lambda \alpha \beta  & \alpha^2& 0 \\ 0 & 0 & \alpha^2\end{pmatrix} = \hat\sol^{V_\lambda,0}_{0,1}(\mathrm{e}_z,\mathrm{e}_z)   \hat\sol^{U_\lambda,0}_{0,1}(\mathrm{e}_z,\mathrm{e}_z).\]
    Thus it suffices to show that for any $k \in \Z^3$ if $k \ne \mathrm{e}_z$,
    \[ \hat\sol^{V_\lambda,0}_{0,1}(\mathrm{e}_z,k)   \hat\sol^{U_\lambda,0}_{0,1}(k,\mathrm{e}_z) = 0.\]
    In particular, we will show that if $\mathrm{e}_x \cdot (k-\mathrm{e}_z) \ne 0$ or $\mathrm{e}_z \cdot (k-\mathrm{e}_z) \ne 0 $, then $\hat\sol^{V_\lambda,0}_{0,1}(\mathrm{e}_z,k) =0$ and if $\mathrm{e}_y \cdot (k-\mathrm{e}_z)\ne 0$, then $\hat\sol^{U_\lambda,0}_{0,1}(k,\mathrm{e}_z) = 0.$

    Suppose first that $\mathrm{e}_x \cdot (k-\mathrm{e}_z)$. Note that for $s \in \R$, $\tau_{s \mathrm{e}_x} V_\lambda = V_\lambda$. Thus using Lemma~\ref{lem.translation-fourier}, we have
    \[e^{2\pi i s(k-\mathrm{e}_z) \cdot \mathrm{e}_x} \hat\sol^{V_\lambda,0}_{0,1}(\mathrm{e}_z,k) =\hat\sol^{\tau_{s\mathrm{e}_x} V_\lambda,0}_{0,1}(\mathrm{e}_z,k) =\hat\sol^{ V_\lambda,0}_{0,1}(\mathrm{e}_z,k).\]
    Then since by assumption $(k-\mathrm{e}_z) \cdot \mathrm{e}_x \ne 0$, this implies $\hat\sol^{V_\lambda,0}_{0,1}(\mathrm{e}_z,k) = 0$, as claimed.

    The remaining two cases follow similarly, using that $\tau_{s\mathrm{e}_z} V_\lambda = V_\lambda$ and $\tau_{s\mathrm{e}_y} U_\lambda = U_\lambda$.
\end{proof}

We note the following direct computation of the spectral properties of the above matrix.

\begin{lemma}
    The matrix
    \[\begin{pmatrix} \alpha^2 + \lambda^2 \beta^2 & -i\lambda \alpha \beta & 0 \\ i \lambda \alpha \beta  & \alpha^2& 0 \\ 0 & 0 & \alpha^2\end{pmatrix}\]
    has eigenvalues
    \[\frac{1}{2}  \beta^2 \lambda^2+  \alpha^2 +\frac{1}{2} |\beta \lambda| \sqrt{\beta^2 \lambda^2 + 4\alpha^2};\quad \frac{1}{2} \beta^2 \lambda^2+  \alpha^2 -\frac{1}{2} |\beta \lambda|  \sqrt{\beta^2 \lambda^2 + 4\alpha^2};\quad \alpha^2,\]
    with corresponding eigenvectors
    \[\begin{pmatrix}
        \frac{\beta^2 \lambda^2 +|\beta \lambda|  \sqrt{\beta^2 \lambda^2 + 4\alpha^2}}{2\alpha \beta \lambda}
        \\i\\0
    \end{pmatrix}; \quad \begin{pmatrix}
        \frac{\beta^2 \lambda^2 -|\beta \lambda|  \sqrt{\beta^2 \lambda^2 + 4\alpha^2}}{2\alpha \beta \lambda}
        \\i\\0
    \end{pmatrix};\quad \begin{pmatrix}
        0\\0\\1
    \end{pmatrix}.\]
\end{lemma}

We are now ready to prove Proposition~\ref{prop.is-controls}.

\begin{proof}[Proof of Proposition~\ref{prop.is-controls}]
    We choose 
    \[u_1 := W_R \quad \text{and} \quad u_2 := W_{-R}\]
    where $R>0$ is chosen so that 
    \[\frac{1}{2}  \beta^2 R^2+  \alpha^2 +\frac{1}{2} |\beta R| \sqrt{\beta^2 R^2 + 4\alpha^2} > e\]
    and so that $A_1,A_2$, defined as in Proposition~\ref{prop.is-controls}, have simple eigenvalues. We have then already guaranteed Item~\ref{item.evals-big}. For Item~\ref{item.span-cond}, since $A_1, A_2$ are self-adjoint, and thus their eigenvectors are orthogonal, it suffices that if $v \in \C^3$, $v \ne 0$, is such that $\mathrm{e_z} \cdot v =0$, then 
    \[v \cdot \begin{pmatrix}
        \frac{\beta^2 R^2 +|\beta \lambda|  \sqrt{\beta^2 R^2 + 4\alpha^2}}{2\alpha \beta R}
        \\i\\0
    \end{pmatrix} \ne 0 \quad \text{or} \quad v \cdot \begin{pmatrix}
         -\frac{\beta^2 R^2 +|\beta \lambda|  \sqrt{\beta^2 R^2 + 4\alpha^2}}{2\alpha \beta R}
        \\i\\0
    \end{pmatrix} \ne 0.\]
    This is however direct from the linear independence of the two vectors above. Thus we conclude.
\end{proof}

\section{Proof of Proposition~\ref{prop.doesnt-vanish}}

The following proof essentially follows as in the scalar advection-diffusion case by directly applying energy estimates to $\|b\|_{L^2}$ as well as the ``projective $H^1$ norm'', $\frac{\|\nabla b\|_{L^2}}{\|b\|_{L^2}}.$

\begin{proof}[Proof of Proposition~\ref{prop.doesnt-vanish}]
    We first note that
    \begin{equation}
    \label{eq.lower-bound-prim}
    \frac{d}{dt} \|b\|_{L^2_x}^2 \geq - 2 \kappa \|\nabla b\|_{L^2_x}^2 -2\|b\|_{L^2_x}^2 \|\nabla u\|_{L^\infty_x} = \bigg({-2} \kappa \frac{\|\nabla b\|_{L^2_x}^2}{\|b\|_{L^2_x}^2} -2 \|\nabla u\|_{L^\infty_x}\bigg)\|b\|_{L^2_x}^2,
    \end{equation}
    so by Gr\"onwall's inequality.
    \begin{equation}
    \label{eq.lower-bound}
        \|b\|_{L^2_x}^2(t) \geq \exp\bigg({-2}\int_0^t \kappa \frac{\|\nabla b\|_{L^2_x}^2}{\|b\|_{L^2_x}^2}(s) + \|\nabla u\|_{L^\infty_x}(s)\,ds\bigg) \|b(0,\cdot)\|_{L^2_x}.
    \end{equation}
    
    So our goal now is to upper bound the growth of $\frac{\|\nabla b\|_{L^2_x}^2}{\|b\|_{L^2_x}^2}.$ Differentiating the equation, we see that
    \[\partial_t \nabla b - \kappa \Delta \nabla b + u \cdot \nabla \nabla b + \nabla u \cdot \nabla b - \nabla b \cdot \nabla u - b \cdot \nabla \nabla u =0.\]
    Therefore,
    \[\frac{d}{dt} \|\nabla b\|_{L^2_x}^2 \leq - 2\kappa \|\nabla^2 b\|_{L^2_x}^2 + 4\|\nabla u\|_{L^\infty_x} \|\nabla b\|_{L^2_x}^2 + C  \|\nabla b\|_{L^2_x}^2 \|\nabla^2 u\|_{L^\infty_x}.\]
    Thus
    \begin{equation}
    \label{eq.proj-upper-prim}
    \frac{d}{dt} \frac{\|\nabla b\|_{L^2_x}^2}{\|b\|_{L^2_x}^2} =  \frac{\frac{d}{dt}\|\nabla b\|_{L^2_x}^2}{\|b\|_{L^2_x}^2} -  \frac{\|\nabla b\|_{L^2_x}^2 \frac{d}{dt} \|b\|_{L^2_x}^2}{\|b\|_{L^2_x}^4} \leq  C\|\nabla^2 u\|_{L^\infty_x}\frac{\|\nabla b\|_{L^2_x}^2}{\|b\|_{L^2_x}^2} - 2\kappa \bigg(\frac{\|\nabla^2 b\|_{L^2_x}^2 }{\|b\|_{L^2_x}^2} - \frac{\|\nabla b\|_{L^2_x}^4}{\|b\|_{L^2_x}^4}\bigg),
    \end{equation}
    where we use the lower bound from~\eqref{eq.lower-bound-prim} and also the bound $\|\nabla u\|_{L^\infty} \leq C\|\nabla^2 u\|_{L^\infty}$.
    
    Then by interpolation
    \[-\|\nabla^2 b\|_{L^2_x}^2 \leq - \frac{\|\nabla b\|_{L^2_x}^4}{\|b\|_{L^2_x}^2},\]
    and so the final term of~\eqref{eq.proj-upper-prim} is non-positive. Discarding it, we then get
    \[\frac{d}{dt} \frac{\|\nabla b\|_{L^2_x}^2}{\|b\|_{L^2_x}^2} \leq C\|\nabla^2 u\|_{L^\infty_x}\frac{\|\nabla b\|_{L^2_x}^2}{\|b\|_{L^2_x}^2}.\]
    Thus by Gr\"onwall's inequality,
    \begin{equation}
    \label{eq.upper-bound}
        \frac{\|\nabla b(t,\cdot)\|_{L^2_x}^2}{\|b(t,\cdot)\|_{L^2_x}^2} \leq \exp\Big(C\int_0^t \|\nabla^2 u\|_{L^\infty_x}(s)\,ds\Big)\frac{\|\nabla b(0,\cdot)\|_{L^2_x}^2}{\|b(0,\cdot)\|_{L^2_x}^2}.
    \end{equation}
    Combining~\eqref{eq.lower-bound} and~\eqref{eq.upper-bound}, we conclude.
\end{proof}

\section{Proof of Proposition~\ref{prop.is-transitive}}

To prove Proposition~\ref{prop.is-transitive}, we essentially want to just take the Fourier transform of~\eqref{eq.induction-equation} and note that pure Fourier mode advecting flows directly ``couple'' Fourier modes of the magnetic field, allowing us to move mass from any Fourier mode onto a Fourier mode with unit wavenumber. However some additional care needs to be taken. First we need to ensure the advecting flow is $\R^3$-valued, so we cannot take a pure Fourier mode but must rather use two Fourier modes. Second, we need to ensure the $u$ is divergence-free, which constrains the directions $\mu \in \C^3$ we can put on the Fourier mode. Finally, we need to ensure that $u$ is compactly supported in time away from $0$; thus we need to use a smooth time cutoff, which makes it easier to work with the ``mild form'' of the equation. It is for these reasons this rather simple fact becomes the somewhat complicated computation below.

\begin{proof}[Proof of Proposition~\ref{prop.is-transitive}]
    As $b_0 \in L^2_{\mathrm{sol},0}(\T^3)$ and $b_0 \ne 0$, we have that there exists some $j \in \Z^3, j \ne 0$ such that $\hat b_0(j) =: w \ne 0$. By the divergence-free condition, we have that $j \cdot w =0$. Let us assume without loss of generality that $\mathrm{e}_z \cdot w \ne 0$. We can also assume without loss of generality that $j \ne \mathrm{e}_z$, as in that case we can conclude with $u=0.$ Then let $v \in \C^3, |v|=1$ such that $w \cdot v  = \mathrm{e}_z \cdot v =0$ and define
    \[\mu := v - \frac{ v \cdot (\mathrm{e}_z-j)}{|\mathrm{e}_z-j|^2} (\mathrm{e}_z-j).\]
    Fix $\phi \in C_c^\infty(0,1)$ such that $\phi \geq 0$ and $\int \phi(t)\,dt = 1$. Then for any $\ep \in \R$, let
    \[u^\ep(t,r) := \ep \phi(t) \Big(\mu e^{2\pi i (\mathrm{e}_z-j)\cdot r} + \overline{\mu} e^{-2\pi i (\mathrm{e}_z-j) \cdot r}\Big).\]
    One can readily verify that $u^\ep : [0,\infty) \times \T^3 \to \R^3$ and $\nabla \cdot u^\ep = 0$. 

    We then claim that for $\ep>0$ small enough, we have that 
    \[v\cdot\hat \sol^{u^\ep, \kappa}_{0,1}(\mathrm{e}_z,j)w \ne 0.\]
    Before proving this claim, let us quickly see that claim implies the proposition. Note that
    \[\|\partial_t^\ell \partial_x^\alpha u^\ep\|_{L^\infty_{t,x}} \leq \ep \|\partial_t^\ell \phi\|_{L^\infty} (C|j|)^{|\alpha|} \leq \ep C(|j|) \|\partial_t^\ell \phi\|_{L^\infty} e^{|\alpha|^2},\]
    so choosing $\ep(|j|)$ small enough, we can ensure that $v\cdot\hat \sol^{u^\ep, \kappa}_{0,1}(\mathrm{e}_z,j)w \ne 0$ while $u^\ep$ obeys bounds independent of $b_0, \kappa$. Then we note that by Lemma~\ref{lem.translation-fourier},
    \[\int e^{-2\pi i (j - \mathrm{e}_z) \cdot y} \widehat{\sol_{0,1}^{\tau_y u_\ep, \kappa}b_0}(\mathrm{e}_z)\,dy = \sol^{u^\ep, \kappa}_{0,1}(\mathrm{e}_z,j) \hat b_0(j) = \sol^{u^\ep, \kappa}_{0,1}(\mathrm{e}_z,j)w \ne 0.\]
    Thus we can choose some $y \in \T^3$ and let $u := \tau_y u_\ep$ so that 
    \[\widehat{\sol_{0,1}^{u, \kappa}b_0}(\mathrm{e}_z) \ne 0,\]
    as desired. This property persists for $t \geq T:=1$ by the explicit form of the heat equation evolution.

    Thus we just need to show that for $\ep>0$ small enough, we have that 
    \[v\cdot\hat \sol^{u^\ep, \kappa}_{0,1}(\mathrm{e}_z,j)w \ne 0.\]
    This in turn is implied by the claim that
    \[\frac{d}{d\ep}\Big|_{\ep =0} v\cdot\hat \sol^{u^\ep, \kappa}_{0,1}(\mathrm{e}_z,j)w \ne 0.\]
    To see this, we note that $\hat \sol^{u^\ep, \kappa}_{0,1}(\mathrm{e}_z,j)w = \hat\beta ^\ep(1,\mathrm{e}_z)$ where $\beta^\ep(t,x)$ solves
    \[\begin{cases}
        \partial_t \beta^\ep - \kappa \Delta \beta^\ep + u^\ep \cdot \nabla \beta^\ep - \beta^\ep \cdot \nabla u^\ep =0\\
        \beta^\ep(0,\cdot) = w e^{2\pi i j \cdot r}.
    \end{cases}\]
    Then $\beta^\ep$ solves the integral equation
    \[\beta^\ep(t) = e^{\kappa t \Delta} \big( w e^{2\pi i j \cdot r}\big) + \ep\int_0^t e^{\kappa (t-s)\Delta}\Big(\beta^\ep(s) \cdot \nabla u^1(s)-u^1(s) \cdot \nabla \beta^\ep(s)\Big)\,ds.\]
    Thus 
    \[\frac{d}{d\ep}\Big|_{\ep=0} \beta^\ep(1) = e^{\kappa \Delta}\big(w e^{2\pi i j \cdot r}\big) + \int_0^1 e^{\kappa(t-s) \Delta} \Big(\beta^0(s) \cdot \nabla u^1(s) - u^1(s) \cdot\nabla \beta^0(s)\Big)\,ds,\]
    where
    \[\beta^0(t) = e^{\kappa t \Delta} \big(w e^{2\pi i j \cdot r}\big) = e^{- 4\pi^2 |j|^2 \kappa t} w e^{2\pi i j \cdot r}.\]
    Plugging in, we thus see that 
  \[\frac{d}{d\ep}\Big|_{\ep=0} \beta^\ep(1) = e^{- 4\pi^2 |j|^2 \kappa} w e^{2\pi i j \cdot r} + \int_0^1 e^{\kappa(t-s) \Delta} e^{- 4\pi^2 |j|^2 \kappa s}\Big( e^{2\pi i j \cdot r} w\cdot \nabla u^1(s) - u^1(s) \cdot\nabla  e^{2\pi i j \cdot r}w\Big)\,ds.\]
    Thus, computing directly,
    \begin{align*}\frac{d}{d\ep}\Big|_{\ep=0} \hat \beta^\ep(1, \mathrm{e}_z) &= 2\pi i\int_0^1 e^{-4\pi^2\kappa(t-s)}  e^{- 4\pi^2 |j|^2 \kappa s} \phi(s) \,ds 
    \\&\qquad\times\bigg(\int e^{4\pi i (j - \mathrm{e}_z) \cdot r}\Big({-w} \cdot (\mathrm{e}_z - j)\overline{\mu} -  \overline{\mu} \cdot  jw\Big)+ w\cdot (\mathrm{e}_z - j)\mu  -  \mu \cdot  jw\,dr\bigg)
    \\&=K\Big( w\cdot (\mathrm{e}_z - j)\mu  -  \mu \cdot  jw\Big),
      \end{align*}
    for some non-zero constant $K$. Recalling that $w \cdot j = 0,$ $\mu \cdot (\mathrm{e}_z - j) =0$, and $v \cdot w =0$, we get that 
    \[v \cdot \frac{d}{d\ep}\Big|_{\ep=0} \hat \beta^\ep(1, \mathrm{e}_z) = K \Big( w \cdot \mathrm{e}_z v \cdot\mu - \mu \cdot \mathrm{e}_z v \cdot w\Big) = K w \cdot \mathrm{e_z} v \cdot \mu.\]
    Then by assumption $w \cdot \mathrm{e}_z \ne 0$ and by construction $v \cdot \mu = \mu \cdot \mu$, so $v \cdot \frac{d}{d\ep}\Big|_{\ep=0} \hat \beta^\ep(1, \mathrm{e}_z) \ne 0$---and we can conclude---provided $\mu \ne 0$. Then we note that by construction $\mu = 0$ if and only if $(\mathrm{e_z} -j) \cdot w = (\mathrm{e_z} -j) \cdot \mathrm{e_z} =0.$ But since $w \cdot j =0$ and $\mathrm{e_z} \cdot w \ne0$, this is impossible, so $\mu \ne 0$ and we conclude.
\end{proof}

{\small
\bibliographystyle{alpha}
\bibliography{references}
}

\end{document}